\documentclass[12pt]{amsart}

\usepackage{amsfonts, amsthm, amsmath, amssymb}
\usepackage{hyperref}
\hypersetup{colorlinks=false}

\usepackage[margin=1.25in]{geometry}

\usepackage{helvet}

\newtheorem{theorem}{Theorem}
\newtheorem{Lemma}{Lemma}

\theoremstyle{definition}

\newtheorem{remark}{Remark}

\renewcommand{\d}{\mathrm{d}}
\renewcommand{\phi}{\varphi}
\renewcommand{\rho}{\varrho}

\title[Subconvexity for $GL(3)\times GL(2)$ $L$-functions]{Subconvexity for $GL(3)\times GL(2)$ $L$-functions in $t$-aspect}

\author{Ritabrata\ Munshi}
\address{School of Mathematics\\ 
Tata Institute of Fundamental Research\\
1 Homi Bhabha Road\\ Colaba\\Mumbai 400005\\ India}
\curraddr{Statistics and Mathematics Unit\\ Indian Statistical Institute\\ 203 B.T. Road\\ Kolkata 700108\\ India} 
\email{rmunshi@math.tifr.res.in, ritabratamunshi@gmail.com}
\subjclass[2010]{11F66}
\keywords{subconvexity, Rankin-Selberg $L$-functions, $GL(3)$ Maass forms}

\begin{document}

\begin{abstract}
Let $\pi$ be a Hecke-Maass cusp form for $SL(3,\mathbb Z)$ and $f$ be a holomorphic (or Maass) Hecke form for $SL(2,\mathbb{Z})$. In this paper we prove the following subconvex bound
$$
L\left(\tfrac{1}{2}+it,\pi\times f\right)\ll_{\pi,f,\varepsilon} (1+|t|)^{\frac{3}{2}-\frac{1}{42}+\varepsilon}.
$$
\end{abstract}

\maketitle

%\tableofcontents
%======================================================================================================================
%======================================================================================================================

\section{Introduction}
\label{intro}

For $\pi$ a Hecke-Maass cusp form for $SL(3,\mathbb Z)$, and $f$ a holomorphic Hecke cusp form for $SL(2,\mathbb{Z})$ the associated Rankin-Selberg $L$-series is given by %the Dirichlet series
$$
L(s,\pi\times f)=\mathop{\sum\sum}_{n,r=1}^\infty \frac{\lambda_\pi(n,r)\lambda_f(n)}{(nr^2)^{s}},
$$
%involving the normalised Fourier coefficients of the forms, 
in the half plane $\sigma>1$. (Here $\lambda_\pi$ and $\lambda_f$ are the normalized Fourier coefficients of the forms.) This series extends to an entire function and satisfies a functional equation of the Riemann type $s\mapsto 1-s$ with a gamma factor of `degree six'. This particular $L$-function plays a crucial role in quantum chaos (see \cite{Sar}), and hence it is important to study its deeper analytic properties. In particular one seeks to understand the size of these functions inside the critical strip. A standard consequence of the functional equation is the easy convexity bound 
$$
L\left(\tfrac{1}{2}+it,\pi\times f\right)\ll_{\pi,f,\varepsilon} (1+|t|)^{\frac{3}{2}+\varepsilon}.
$$
The Lindel\"{o}f hypothesis predicts that such a bound holds with any positive exponent in place of $3/2+\varepsilon$. But even breaking the convexity barrier is hard and has remained open so far. The purpose of this paper is to prove the following subconvex bound.

\begin{theorem}
\label{mthm}
Let $\pi$ be a Hecke-Maass cusp form for $SL(3,\mathbb Z)$, and $f$ a holomorphic Hecke cusp form for $SL(2,\mathbb{Z})$. Then we have
$$
L\left(\tfrac{1}{2}+it,\pi\times f\right)\ll_{\pi,f,\varepsilon} (1+|t|)^{\frac{3}{2}-\frac{1}{42}+\varepsilon}.
$$
\end{theorem}

Subconvex bounds in the $t$-aspect are known for $L$-functions of degree upto three over the field of rationals (see \cite{W}, \cite{Go} and \cite{Mu-JAMS}). Similar bounds are also known for the Rankin-Selberg $L$-function $L(s,f\times g)$ for two $GL(2)$ forms $f$ and $g$. The $t$-aspect subconvexity for genuine $GL(4)$ $L$-functions remains an important open problem.  Our method of proof is similar to the one given in \cite{Mu-JAMS} and is based on the separation of oscillation technique (as introduced in \cite{Mu}). The key reason for a similar argument to be effective here is the following observation 
$$
\sideset{}{^\star}\sum_{a\bmod{q}} S(\bar{a},n;q)e(\bar{a}m/q)\leadsto qe(-\bar{m}n/q).
$$ 
In other words, the $GL(3)$, $GL(2)$ Voronoi summations together transform the Ramanujan sums $\sum_a^\star e(a(n-m)/q)$ in the delta method to additive characters with respect to the $GL(3)$ variable. As such we save more by applying Poisson summation after Cauchy's inequality. This is the vital structural input in this paper. The same feature helps us to prove a subconvex bound for these $L$-functions in the $GL(2)$ spectral aspect. This will be taken up in another paper. Let us also note that our argument works for Maass forms $f$, after mild alterations. In fact the argument can be extended to Rankin-Selberg convolutions of a general $GL(3)$ and a general $GL(2)$ automorphic forms over $\mathbb{Q}$. 

The main technical heart of \cite{Mu-JAMS} was the analysis of the integral transforms. In this paper we give a simpler analysis of these integrals. This is very much desired as the technique of \cite{Mu-JAMS} leads to the Weyl bound in the case of $GL(2)$ and $GL(1)$ $L$-functions (see \cite{AS}), and now perhaps with this simplification one can go further. 

%\ack Thanks...

%====================

\section{The set-up}

Let $\lambda_\pi(n,m)$ denote the normalised Fourier coefficients of the form $\pi$ (see Chapter~6 of \cite{G}) and let $\lambda_f(n)$ denote the normalised Fourier coefficients of the form $f$ (see \cite{IK}). Suppose $t>2$, then by approximate functional equation (see \cite{IK}) we have
\begin{align}
\label{afe}
L\left(\tfrac{1}{2}+it,\pi\times f\right)\ll t^{\varepsilon}\; \mathop{\sup}_{N\leq t^{3+\varepsilon}} \frac{|S(N)|}{N^{1/2}}+t^{-2018}
\end{align}
where $S(N)$ is a sum of type 
$$
S(N):=\mathop{\sum\sum}_{n,r=1}^\infty \lambda_\pi(n,r)\lambda_f(n) (nr^2)^{-it}V\left(\frac{nr^2}{N}\right)
$$
for some smooth function $V$ supported in $[1,2]$ and satisfying $V^{(j)}(x)\ll_j 1$.

\begin{remark}[Notation]
In this paper the notation $\alpha\ll A$ will mean that for any $\varepsilon>0$, there is a constant $c$ such that $|\alpha|\leq cA t^\varepsilon$. The dependence of the constant on $\pi$, $f$ and $\varepsilon$, when occurring, will be ignored.
\end{remark}

 Using the Ramanujan bound on average , i.e.
$$
\mathop{\sum\sum}_{n_1^2n_2\leq x}|\lambda(n_1,n_2)|^2\ll x^{1+\varepsilon},
$$ 
we further conclude that 
\begin{align}
\label{new-afe}
L\left(\tfrac{1}{2}+it,\pi\times f\right)\ll \mathop{\sup}_{r\leq t^\theta} \mathop{\sup}_{\frac{t^{3-\theta}}{r^2}\leq N\leq \frac{t^{3+\varepsilon}}{r^2}} \frac{|S_r(N)|}{N^{1/2}}+t^{(3-\theta)/2}
\end{align}
where 
$$
S_r(N):=\mathop{\sum}_{n=1}^\infty \lambda_\pi(n,r)\lambda_f(n) n^{-it}V\left(\frac{n}{N}\right)
$$  Hence to establish subconvexity we need to show cancellation in the sum $S_r(N)$ for $N$ roughly of size $t^{3}$ and $r$ small. We can and shall further normalize $V$, for convenience, so that $\int V(y)\mathrm{d}y=1$.

\subsection{The delta method}
There are three oscillatory factors contributing to the sum $S_r(N)$. Our method is based on separating these oscillations using the circle method. In the present situation we will use a version of the delta method of Duke, Friedlander and Iwaniec. More specifically we will use the expansion (20.157) given in Chapter~20 of \cite{IK}. Let $\delta:\mathbb{Z}\rightarrow \{0,1\}$ be defined by
$$
\delta(n)=\begin{cases} 1&\text{if}\;\;n=0;\\
0&\text{otherwise}.\end{cases}
$$
We seek a Fourier expansion which matches with $\delta$ in the range $[-2M,2M]$. For this we pick $Q=2M^{1/2}$. Then we have
\begin{align}
\label{cm}
\delta(n)=\frac{1}{Q}\sum_{1\leq q\leq Q} \;\frac{1}{q}\; \sideset{}{^\star}\sum_{a\bmod{q}}e\left(\frac{na}{q}\right) \int_\mathbb{R}g(q,x) e\left(\frac{nx}{qQ}\right)\mathrm{d}x
\end{align}
for $n\in\mathbb Z\cap [-2M,2M]$ (and $e(z)=e^{2\pi iz}$). The $\star$ on the sum indicates that the sum over $a$ is restricted by the condition $(a,q)=1$. The function $g$ is the only part in the formula which is not explicitly given. We only need the following two properties (see (20.158) and (20.159) of \cite{IK})
\begin{align}
\label{g-h}
g(q,x) &=1+h(q,x),\;\;\;\text{with}\;\;\;h(q,x)=O\left(\frac{1}{qQ}\left(\frac{q}{Q}+|x|\right)^A\right),\\
\nonumber g(q,x)&\ll |x|^{-A}
\end{align}
for any $A>1$. In particular the second property implies that the effective range of the integral in \eqref{cm} is $[-M^\varepsilon, M^\varepsilon]$.

\subsection{Separation of oscillation}
We apply \eqref{cm} directly to $S_r(N)$ as a device to separate the oscillations of $\lambda(n,r)$ and $\lambda_f(n)n^{-it}$. This by itself does not suffice, and as in \cite{Mu0} and \cite{Mu-JAMS} we need a `conductor lowering mechanism'. For this purpose we introduce an extra integral namely
$$
S_r(N)=\frac{1}{K}\int_{\mathbb R}V\left(\frac{v}{K}\right)\mathop{\sum\sum}_{\substack{n,m=1\\n=m}}^\infty \lambda_\pi(n,r)\lambda_f(m)m^{-it}\:\left(\frac{n}{m}\right)^{iv}V\left(\frac{n}{N}\right)U\left(\frac{m}{N}\right)\mathrm{d}v,
$$
where $t^\varepsilon<K<t^{1-\varepsilon}$ is a parameter which will be chosen optimally later, and  $U$ is a smooth function supported in $[1/2,5/2]$, with $U(x)=1$ for $x\in[1,2]$ and $U^{ (j)}\ll_j 1$. For $n,m\asymp N$, the integral
$$
\frac{1}{K}\int_{\mathbb R}V\left(\frac{v}{K}\right)\left(\frac{n}{m}\right)^{iv}\mathrm{d}v
$$
is negligibly small (i.e. $O_A(t^{-A})$ for any $A>0$) if $|n-m|\gg Nt^{\varepsilon}/K$. Hence we can apply \eqref{cm} with 
\begin{align}
\label{q-choice}
Q=t^\varepsilon\left(\frac{N}{K}\right)^{1/2}
\end{align}
and we get that upto a negligible error term $S_r(N)$ is given by
\begin{align}
\label{S(N)-cm}
&\frac{1}{QK}\int_{\mathbb{R}}W(x)\int_{\mathbb R}V\left(\frac{v}{K}\right)\sum_{1\leq q\leq Q}\;\frac{g(q,x)}{q}\;\sideset{}{^\star}\sum_{a\bmod{q}} \\
\nonumber &\times \mathop{\sum}_{n=1}^\infty \lambda_\pi(n,r)e\left(\frac{an}{q}\right)e\left(\frac{nx}{qQ}\right) n^{iv} V\left(\frac{n}{N}\right)\\
\nonumber &\times \sum_{m=1}^\infty \lambda_f(m)m^{-i(t+v)}e\left(-\frac{am}{q}\right)e\left(-\frac{mx}{qQ}\right)U\left(\frac{m}{N}\right)\mathrm{d}v\mathrm{d}x,
\end{align} 
where $W$ is a smooth bump function with support $[-t^\varepsilon, t^\varepsilon]$

\subsection{Sketch of proof} We end this section with a brief sketch of the proof. For simplicity let us focus on the generic case, i.e. $N=t^3$, $r=1$ and $q\sim Q=t^{3/2}/K^{1/2}$, so that the main object of study is given by
$$
\mathop{\int}_{v\sim K}\sum_{q\sim Q}\;\sideset{}{^\star}\sum_{a\bmod{q}} \sum_{n\sim N}\lambda_\pi(n,1)e\left(\frac{an}{q}\right)n^{iv}\sum_{m\sim N}\lambda_f(m)e\left(-\frac{am}{q}\right)m^{-i(t+v)}.
$$
Our aim is to save $N$ plus a `little more'. First we apply the Voronoi summation formulae to both the $m$ and $n$ sums. In the $GL(2)$ (resp. $GL(3)$) Voronoi we save $(NK)^{1/2}/t$ (resp. $N^{1/4}/K^{3/4}$) and the dual length becomes $m^\star\sim t^2/K$ (resp. $n^\star \sim K^{3/2}N^{1/2}$). Also we save $\sqrt{Q}$ in the $a$ sum and $\sqrt{K}$ in the $v$ integral. Hence in total we have saved $N/t$, and it remains to save $t$ plus a little extra in a sum of the form
$$
\sum_{q\sim Q}\; \sum_{n\sim K^{3/2}N^{1/2}}\lambda_\pi(1,n)\sum_{m\sim t^2/K}\lambda_f(m)\;\mathfrak{C}\:\mathfrak{I}
$$
where $\mathfrak{I}$ is an integral transform which oscillates like $n^{iK}$ with respect to $n$, and the character sum is given by
$$
\mathfrak{C}=\sideset{}{^\star}\sum_{a\bmod{q}} S(\bar{a},n;q)e\left(\frac{\bar{a}m}{q}\right)\leadsto qe\left(-\frac{\bar{m}n}{q}\right).
$$
Next applying the Cauchy inequality we arrive at 
$$
\sum_{n\sim K^{3/2}N^{1/2}}\Bigl|\sum_{q\sim Q}\;\sum_{m\sim t^2/K}\lambda_f(m)\;e\left(-\frac{\bar{m}n}{q}\right)\:\mathfrak{I}\Bigr|^2
$$
where we seek to save $t^2$ plus extra. Opening the absolute value square we apply the Poisson summation formula on the sum over $n$. We save enough in the zero frequency (diagonal contribution) if $t^2Q/K>t^2$ i.e. if $K<t$. On the other hand we save enough in the non-zero frequencies if $K^{3/2}N^{1/2}/K^{1/2}>t^2$ which boils down to $K>t^{1/2}$. 

\begin{remark}
Notice that since the character sum boils down to an additive character we are saving more than the usual. In the usual case we  would have saved $K^{3/2}N^{1/2}/QK^{1/2}$, which would be larger than $t^2$ only if we had $K>t^{4/3}$. This would contradict the upper bound $K<t$.
\end{remark}

%=====================

\section{Voronoi summation formulae}

\subsection{$GL(2)$ Voronoi}
Consider the sum over $m$ in \eqref{S(N)-cm}. Applying the Voronoi summation formula this transforms into 
$$
\frac{N^{1-i(t-v)}}{q}\sum_{m=1}^\infty \lambda_f(m) e\left(\frac{\bar{a}m}{q}\right) \int_0^\infty U(y) y^{-i(t+v)}e\left(-\frac{Nxy}{qQ}\right)J_{k-1}\left(\frac{4\pi\sqrt{mNy}}{q}\right)\mathrm{d}y
$$
where $k$ is the weight of the form $f$. Extracting the oscillation of the Bessel function we see that the above sum is essentially given by a sum of two terms of the form
\begin{align}
\label{voronoi2}
\frac{N^{3/4-i(t-v)}}{q^{1/2}}\sum_{m=1}^\infty \frac{\lambda_f(m)}{m^{1/4}} e\left(\frac{\bar{a}m}{q}\right) \int_0^\infty U(y) y^{-i(t+v)}e\left(-\frac{Nxy}{qQ}\pm \frac{2\sqrt{mNy}}{q}\right)\mathrm{d}y.
\end{align}
By repeated integration by parts it follows that the integral is negligibly small if $m\gg t^\varepsilon \max\{K,t^2q^2/N\}=:M_0$. In the complementary range the size of the integral is given by the second derivative bound. However we need a more precise analysis of the integral based on the stationary phase expansion. In particular we note that when $Nx/qQ\ll t^{1-\varepsilon}$ then $m\asymp (qt)^2/N$, otherwise the integral is negligibly small.

\subsection{$GL(3)$ Voronoi}
Next we apply the $GL(3)$ Voronoi summation to the sum over $n$ in \eqref{S(N)-cm}. A similar sum occurred in \cite{Mu-JAMS}. The only difference is that there we had $r=1$, but here $r$ is allowed to take small values $r\ll t^\theta$. This only introduces certain cosmetic complications. Let $\{\alpha_i:i=1,2,3\}$ be the Langlands parameters for $\pi$. Let $g$ be a compactly supported smooth function on $(0,\infty)$. We define for $\ell=0,1$
$$
\gamma_\ell(s):=\frac{\pi^{-3s-\frac{3}{2}}}{2}\:\prod_{i=1}^3\frac{\Gamma\left(\frac{1+s+\alpha_i+\ell}{2}\right)}{\Gamma\left(\frac{-s-\alpha_i+\ell}{2}\right)},
$$
set $\gamma_\pm(s)=\gamma_0(s)\mp i\gamma_1(s)$ and let
$$
G_{\pm}(y)=\frac{1}{2\pi i}\int_{(\sigma)}y^{-s}\gamma_\pm(s)\tilde g(-s)ds,
$$
where
$\sigma>-1+\max\{-\text{Re}(\alpha_1),-\text{Re}(\alpha_2),-\text{Re}(\alpha_3)\}$. The $GL(3)$ Voronoi summation formula (see \cite{MS}) is given by
\begin{align*}
&\sum_{n=1}^\infty \lambda_\pi(n,r)e\left(\frac{an}{q}\right)g(n)\\
\nonumber =&q\sum_\pm\sum_{n_1|qr}\sum_{n_2=1}^\infty \frac{\lambda_\pi(n_1,n_2)}{n_1n_2}S(r\bar a, \pm n_2; qr/n_1)G_\pm\left(\frac{n_1^2n_2}{q^3r}\right).
\end{align*}
In the present case we have $g(n)=e\left(nx/qQ\right) n^{iv} V\left(n/N\right)$. Extracting the oscillation of the integral transform (see e.g. Lemma~2.1 of \cite{L}), as in the case of $GL(2)$ above, we essentially arrive at the following expression
\begin{align}
\label{voronoi3}
\frac{N^{2/3+iv}}{qr^{2/3}}\sum_\pm\sum_{n_1|qr}n_1^{1/3}&\sum_{n_2=1}^\infty \frac{\lambda_\pi(n_1,n_2)}{n_2^{1/3}}S(r\bar a, \pm n_2; qr/n_1)\\
\nonumber &\times \int_0^\infty V(z)z^{iv}e\left(\frac{Nxz}{qQ}\pm \frac{3(Nn_1^2n_2z)^{1/3}}{qr^{1/3}}\right)\mathrm{d}z.
\end{align}
By repeated integration by parts we see that the integral is negligibly small if $n_1^2n_2\gg t^\varepsilon((qK)^3r/N+K^{3/2}N^{1/2}rx^3)=:N_0$. We now substitute \eqref{voronoi2} in place of the third line and  \eqref{voronoi3} in place of the second  line of \eqref{S(N)-cm}, to get the object of focus. \\

%==============================
\section{Reduction of integrals}

\subsection{Simplifying the integrals}
We have transformed the sum in \eqref{S(N)-cm} into a new object with four integrals, which we need to simplify. Consider the integral over $x$ which boils down to 
$$
\int_\mathbb{R} W\left(x\right) g(q,x) e\left(\frac{Nx(z-y)}{qQ}\right)\mathrm{d}x.
$$
Using \eqref{g-h} this splits as the sum of two integrals
$$
\int_\mathbb{R} W\left(x\right) e\left(\frac{Nx(z-y)}{qQ}\right)\mathrm{d}x+
\int_\mathbb{R} W\left(x\right) h(q,x) e\left(\frac{Nx(z-y)}{qQ}\right)\mathrm{d}x,
$$
where in the second integral the weight function $h$ has smaller size. In the first integral by repeated integration by parts we see that it is negligibly small unless $|z-y|\ll t^\varepsilon q/QK$. (We will continue our analysis with the first integral. For the second integral, apart from the fact that the weight function $h$ is of size $1/qQ$, we are able to get a weaker restriction $|z-y|\ll t^\varepsilon/K$ by considering the $v$ integral. As such we obtain much better final bound in this case.)   Writing $z=y+u$ with $|u|\ll t^\varepsilon q/QK$ we arrive at the $y$ integral
\begin{align}
\label{integral}
I(m,n_1^2n_2,q):=\int_0^\infty U(y)y^{-it}e\left(\pm \frac{2\sqrt{mNy}}{q}\pm \frac{3(Nn_1^2n_2(y+u))^{1/3}}{qr^{1/3}}\right)\mathrm{d}y.
\end{align}

\subsection{Size of the integral $I(\dots)$}
\label{integral-norm}
Suppose $K=t^{1-\eta}$ for some $\eta>0$, then we claim that we essentially have $I(\dots)\ll t^{-1/2}$. We will prove that the bound holds in $L^2$ sense.

\begin{Lemma}
\label{lem:integral}
Let 
$$
L=\int W(w) |I(m,N_0w^3,q)|^2\mathrm{d}w
$$ 
where $W$ is a bump function. Then we have $L\ll 1/t$.
\end{Lemma}

\begin{proof}
To prove this assertion we make a change of variable $z=y^{1/2}$, so that the phase function in \eqref{integral} reduces to
$$
P=-\frac{t}{\pi}\log z \pm \frac{2\sqrt{mN}z}{q}\pm \frac{3(NN_0(z^2+u))^{1/3}w}{qr^{1/3}}.
$$
Then
$$
P''=\frac{t}{\pi z^2}\mp \frac{2(NN_0)^{1/3}w}{3qr^{1/3}z^{4/3}}+\:\text{smaller order terms}.
$$
For this to be smaller than $t$ in magnitude one at least needs a negative sign in the second term and  $3(NN_0)^{1/3}w/qr^{1/3}\asymp t$. 
%Now since $B>0$ and $K=t^{1-\eta}$ the integral in \eqref{voronoi3} is negligibly small unless $x<0$ and $N|x|/qQ\asymp t$. 
Except this case we have $I(\dots)\ll t^{-1/2}$ by the second derivative bound. In the special situation  we have $N_0\asymp (tq)^3r/N$. Opening the absolute value square we arrive at
\begin{align*}
L&\ll \iint U(y_1)U(y_2) \Bigl|\int W(w)e\left(\frac{3w(NN_0)^{1/3}}{qr^{1/3}}((y_1+u)^{1/3}-(y_2+u)^{1/3})\right)\mathrm{d}w\Bigr|\\
& \ll \iint_{|y_1-y_2|\ll 1/t} U(y_1)U(y_2) +t^{-2018}\ll 1/t.
\end{align*}
The lemma follows.
\end{proof}
% take a step back and consider the integrals over $z$, $v$ and $x$. The stationary point of the $z$ integral in \eqref{voronoi3} can be written as $z_0+z_1+z_2+\dots$ with $z_i\ll (K/t)^i$. The leading two terms are seen to be
%$$
%z_0=\left(\frac{BqQ}{3N|x|}\right)^{3/2},\;\;\;z_1=\frac{vqQ}{4\pi N|x|}.
%$$ 
%Now for the $v$, $x$ integral the phase is given by
%$$
%\frac{v}{2\pi}\log z_0 +\frac{vz_1}{2\pi z_0}-\frac{N|x|(z_0+z_1)}{qQ}+Bz_0^{1/3}+\frac{Bz_1}{3z_0^{2/3}}-\frac{v}{2\pi}\log y +\frac{N|x|y}{qQ}
%$$
%plus terms which are of magnitude $O(K^3/t^2)$. Now we write $Kv$ in place of $v$, so that $v\sim 1$. The Hessian matrix (only the leading terms) is given by 
%\begin{align*}
%\begin{pmatrix}
%\frac{B^{3/2}(qQ)^{1/2}}{2(3N)^{1/2}|x|^{5/2}}&\frac{3K}{4\pi|x|^2}\\
%\frac{3K}{4\pi|x|^2}&\frac{3^{3/2}K^2(N|x|)^{1/2}}{4\pi^2 B^{3/2}(qQ)^{1/2}}
%\end{pmatrix}.
%\end{align*}
%Hence it follows that the integral over $(v\sim 1,|x|\sim X)$ is bounded by $X^2/K$. (So effectively in this case we are showing that $I(\dots)\ll QX^2/qt^{1/2}$, where $q\asymp NX/Qt$. So the loss compared to the generic case is given by $Q/q=Q^2t/N=t^\eta$. )

%==========================
\section{Cauchy and Poisson}

\subsection{Cauchy inequality}
The expression in \eqref{S(N)-cm} has essentially reduced to
\begin{align*}
&\frac{N^{5/12}}{r^{2/3}}\sum_{1\leq q\leq Q}\;\frac{1}{q^{3/2}}\;\sideset{}{^\star}\sum_{a\bmod{q}} \\
\nonumber &\times \sum_\pm\sum_{n_1|qr}n_1^{1/3}\sum_{n_2\ll N_0/n_1^2} \frac{\lambda_\pi(n_1,n_2)}{n_2^{1/3}}S(r\bar a, \pm n_2; qr/n_1)\\
\nonumber &\times \sum_{m\ll M_0} \frac{\lambda_f(m)}{m^{1/4}} e\left(\frac{\bar{a}m}{q}\right) \;I(m,n_1^2n_2,q).
\end{align*} 
Splitting $q$ in dyadic blocks $q\sim C$, and writing $q=q_1q_2$ with $q_1|(n_1r)^\infty$, $(n_1r,q_2)=1$, we see that the contribution of the $C$-block to the above sum is dominated by
\begin{align}
\label{S(N)-cm-sum}
\frac{N^{5/12}}{r^{2/3}C^{3/2}}\sum_\pm & \sum_{n_1\ll  Cr}n_1^{1/3}\sum_{\frac{n_1}{(n_1,r)}|q_1|(n_1r)^\infty}\:\sum_{n_2\ll N_0/n_1^2} \frac{|\lambda_\pi(n_1,n_2)|}{n_2^{1/3}}  \\
\nonumber &\times \Bigl|\sum_{q_2\sim C/q_1}\;\sum_{m\ll M_0} \frac{\lambda_f(m)}{m^{1/4}} \;\mathcal{C}(\dots)\:I(m,n_1^2n_2,q)\Bigr|,
\end{align} 
where the character sum $\mathcal{C}(\dots)$ is given by
$$
\sideset{}{^\star}\sum_{a\bmod{q}}\:S(r\bar a, \pm n_2; qr/n_1)e\left(\frac{\bar{a}m}{q}\right)=\sum_{d|q}d\mu\left(\frac{q}{d}\right)\mathop{\sideset{}{^\star}\sum}_{\substack{\alpha\bmod{qr/n_1}\\n_1\alpha\equiv -m\bmod{d}}}e\left(\pm\frac{\bar{\alpha}n_2}{qr/n_1}\right). 
$$
To analyse the sum in \eqref{S(N)-cm-sum} further we break the sum over $m$ into dyadic blocks. Then applying Cauchy's inequality and using the Ramanujan bound on average we see that the expression in \eqref{S(N)-cm-sum} is dominated by
\begin{align}
\label{cauchy}
\sup_{M_1\ll M_0} \frac{N^{5/12}N_0^{1/6}}{r^{2/3}C^{3/2}}\sum_\pm & \sum_{n_1\ll  Cr}\frac{1}{n_1^{1/3}}\sum_{\frac{n_1}{(n_1,r)}|q_1|(n_1r)^\infty}\:\Omega^{1/2}  
\end{align} 
where 
\begin{align}
\label{omega}
\Omega=\sum_{n_2\ll N_0/n_1^2} \Bigl|\sum_{q_2\sim C/q_1}\;\sum_{m\sim M_1} \frac{\lambda_f(m)}{m^{1/4}} \;\mathcal{C}(\dots)\:I(m,n_1^2n_2,q)\Bigr|^2,
\end{align}
and $M_1\ll M_0=K+C^2t^2/N$, $N_0=(CK)^3r/N+K^{3/2}N^{1/2}r$. 

\subsection{Poisson summation}
Smoothing out the outer sum in \eqref{omega}, opening the absolute value square and applying the Poisson summation formula we arrive at
\begin{align}
\label{omega-bound}
\Omega\ll \frac{N_0}{n_1^2 M_1^{1/2}}\mathop{\sum\sum}_{q_2,q_2'\sim C/q_1}\:\mathop{\sum\sum}_{m,m'\sim M_1}\:\sum_{n_2\in \mathbb{Z}}\:|\mathfrak{C}|\:|\mathfrak{I}|,
\end{align}
where
\begin{align*}
\mathfrak{C}=\mathop{\sum\sum}_{\substack{d|q\\d'|q'}}dd'\mu\left(\frac{q}{d}\right)\mu\left(\frac{q'}{d'}\right)\mathop{\mathop{\sideset{}{^\star}\sum}_{\substack{\alpha\bmod{qr/n_1}\\n_1\alpha\equiv -m\bmod{d}}}\;\mathop{\sideset{}{^\star}\sum}_{\substack{\alpha'\bmod{q'r/n_1}\\n_1\alpha'\equiv -m'\bmod{d'}}}}_{\bar{\alpha}q_2'-\bar{\alpha}'q_2\equiv n_2\bmod{q_2q_2'q_1r/n_1}}\; 1,
\end{align*}
and
\begin{align*}
\mathfrak{I}=\int W(w)I(m,N_0w,q)\overline{I(m',N_0w,q')}\:e\left(-\frac{N_0n_1n_2w}{q_2q_2'q_1r}\right)\mathrm{d}w.
\end{align*}
By repeated integration by parts we see that the integral is negligibly small if 
$$
|n_2|\gg t^\varepsilon \frac{CN^{1/3}r^{2/3}}{n_1q_1N_0^{2/3}}:=N_2.
$$
Moreover from our analysis in Subsection~\ref{integral-norm} it follows that $\mathfrak{I}\ll t^{-1}$. 

\subsection{The zero frequency}
The zero frequency $n_2=0$ has to be treated differently. Let $\Omega_0$ denotes the contribution of the zero frequency to $\Omega$, and let $\Sigma_0$ be its contribution to \eqref{cauchy}.

\begin{Lemma}
\label{lem:zero}
We have 
\begin{align*}
\Omega_{0}&\ll \frac{N_0M_1^{1/2}C^2r}{n_1^2q_1 t}\;\left(C+M_1\right),
\end{align*}
and 
\begin{align*}
\Sigma_0\ll r^{1/3}N^{1/2}t^{3/2}\:(t^{-1/2+\eta/2}+t^{-3\eta/2}).
\end{align*}
\end{Lemma}

\begin{proof}
 In the case $n_2=0$ it follows from the congruence conditions that $q_2=q_2'$ and $\alpha=\alpha'$. So the character sum is bounded as 
$$
\mathfrak{C}\ll \mathop{\sum\sum}_{\substack{d, d'|q}}dd'\:\mathop{\sideset{}{^\star}\sum}_{\substack{\alpha\bmod{qr/n_1}\\n_1\alpha\equiv -m\bmod{d}\\n_1\alpha\equiv -m'\bmod{d'}}}\; 1\ll \mathop{\sum\sum}_{\substack{d, d'|q\\(d,d')|(m-m')}}dd'\:\frac{qr}{[d,d']},
$$
and hence we get
\begin{align*}
\Omega_{0}&\ll \frac{N_0}{n_1^2 M_1^{1/2}t}\mathop{\sum}_{q_2\sim C/q_1}\:qr\mathop{\sum\sum}_{\substack{d, d'|q}}(d,d')\mathop{\sum\sum}_{\substack{m,m'\sim M_1\\(d,d')|m-m'}}\:1\\
&\ll \frac{N_0}{n_1^2 M_1^{1/2}t}\mathop{\sum}_{q_2\sim C/q_1}\:qr\mathop{\sum\sum}_{\substack{d, d'|q}} \;\left(M_1(d,d')+M_1^2\right).
\end{align*}
Trivially executing the remaining sums we get the first part of the lemma.

This bound when substituted in place of $\Omega$ in \eqref{cauchy} yields the bound 
\begin{align}
\label{c-deno}
\frac{N^{3/4}Kr^{1/3}}{t^{1/2}}\left(1+\frac{K^{1/2}}{C^{1/2}}+\frac{C^{1/2}t}{N^{1/2}}\right)\left(K^{1/4}+\frac{(Ct)^{1/2}}{N^{1/4}}\right)
\end{align}
Here if we substitute $\sqrt{N/K}$ in place of $C$ and use the fact that $K=t^{1-\eta}$, then we get $O(r^{1/3}N^{1/2}t^{1+\eta/2})$ as the final bound to \eqref{cauchy}. This takes care of all the terms in \eqref{c-deno} except the single term which has $C^{1/2}$ in the denominator. This occurs only when $M_1\sim K$, which is possible only if $N|x|/CQ\sim t$ (as otherwise the integral in \eqref{voronoi2} is negligibly small). In this case we get 
$$
\frac{N^{3/4}Kr^{1/3}}{t^{1/2}}\:\frac{K^{3/4}}{C^{1/2}}\ll \frac{N^{3/4}K^{7/4}r^{1/3}}{t^{1/2}}\:\frac{(Qt)^{1/2}}{(N|x|)^{1/2}}.
$$ 
The integral over $x$ takes care of the $x^{1/2}$ in the denominator, and we see that the total contribution of this term to \eqref{cauchy} is dominated by $O(r^{1/3}N^{1/2}t^{3/2-3\eta/2})$. The lemma follows.
\end{proof}

%==============================
\section{Analysis of non-zero frequencies} 

\subsection{The character sum}
Our next lemma gives a bound for $\mathfrak{C}$.

\begin{Lemma}
\label{lem:char-sum}
We have
$$
\mathfrak{C}\ll \frac{q_1^3r}{n_1}\;\mathop{\sum\sum}_{\substack{d_2|(q_2,q_2'n_1+mn_2)\\ d_2'|(q_2',q_2n_1+m'n_2)}}d_2d_2'.
$$
\end{Lemma}

\begin{proof}
The `character sum' $\mathfrak{C}$ can be dominated by a product of two sums $\mathfrak{C}\ll \mathfrak{C}_1\mathfrak{C}_2$ where
\begin{align*}
\mathfrak{C}_1=\mathop{\sum\sum}_{\substack{d_1, d_1'|q_1}}d_1d_1'\mathop{\mathop{\sideset{}{^\star}\sum}_{\substack{\alpha\bmod{q_1r/n_1}\\n_1\alpha\equiv -m\bmod{d_1}}}\;\mathop{\sideset{}{^\star}\sum}_{\substack{\alpha'\bmod{q_1r/n_1}\\n_1\alpha'\equiv -m'\bmod{d_1'}}}}_{\bar{\alpha}q_2'-\bar{\alpha}'q_2\equiv n_2\bmod{q_1r/n_1}}\; 1,
\end{align*}
and
\begin{align*}
\mathfrak{C}_2=\mathop{\sum\sum}_{\substack{d_2|q_2\\d_2'|q_2'}}d_2d_2'\mathop{\mathop{\sideset{}{^\star}\sum}_{\substack{\alpha\bmod{q_2}\\n_1\alpha\equiv -m\bmod{d_2}}}\;\mathop{\sideset{}{^\star}\sum}_{\substack{\alpha'\bmod{q_2'}\\n_1\alpha'\equiv -m'\bmod{d_2'}}}}_{\bar{\alpha}q_2'-\bar{\alpha}'q_2\equiv n_2\bmod{q_2q_2'}}\; 1.
\end{align*}
In the second sum since $(n_1,q_2q_2')=1$, we get $\alpha\equiv -m\bar{n}_1\bmod{d_2}$ and $\alpha'\equiv -m'\bar{n}_1\bmod{d_2'}$. Then using the congruence modulo $q_2q_2'$ we are able to conclude that 
\begin{align*}
\mathfrak{C}_2\ll \mathop{\sum\sum}_{\substack{d_2|(q_2,q_2'n_1+mn_2)\\ d_2'|(q_2',q_2n_1+m'n_2)}}d_2d_2'.
\end{align*}
In the first sum $\mathfrak{C}_1$ the congruence condition determines $\alpha'$ uniquely in terms of $\alpha$, and hence  
\begin{align*}
\mathfrak{C}_1\ll \mathop{\sum\sum}_{\substack{d_1, d_1'|q_1}}d_1d_1'\mathop{\sideset{}{^\star}\sum}_{\substack{\alpha\bmod{q_1r/n_1}\\n_1\alpha\equiv -m\bmod{d_1}}}\; 1\ll \frac{q_1^3r}{n_1}.
\end{align*}
This completes the proof of the lemma.
\end{proof}

We now substitute these bounds in \eqref{omega-bound}. Writing $q_2d_2$ in place of $q_2$ and $q_2'd_2'$ in place of $q_2'$ we get that the contribution of the non-zero frequencies to $\Omega$ is
\begin{align}
\label{omega-bound-2}
\Omega_{\neq 0}\ll \frac{N_0q_1^3r}{n_1^3M_1^{1/2}}\mathop{\sum\sum}_{d_2,d_2'}\:d_2d_2'\mathop{\sum\sum}_{\substack{q_2\sim C/q_1d_2\\q_2'\sim C/q_1d_2'}}\:\mathop{\mathop{\sum\sum}_{m,m'\sim M_1}\sum_{n_2\in \mathbb{Z}-\{0\}}}_{\substack{q_2'd_2'n_1+mn_2\equiv 0\bmod{d_2}\\ q_2d_2n_1+m'n_2\equiv 0\bmod{d_2'}}}\:|\mathfrak{I}|.
\end{align}
We denote by $\Sigma_{\neq 0}$ the term we get by substituting this for $\Omega$ in \eqref{cauchy}.

\subsection{The case of small modulus} In this section we will consider the case where $q\sim C\ll t^{1+\varepsilon}$. Recall that we have $\mathfrak{I}\ll 1/t$ and $n_2\neq 0$. 

\begin{Lemma}
\label{lem:small-c}
The contribution of  $q\sim C\ll t^{1+\varepsilon}$, and $n_2\neq 0$  to \eqref{cauchy} is bounded by
$$
\Sigma_{\neq 0, \:\text{small}}\ll r^{1/2}t^{3/2}N^{1/2}\left(\frac{t^{3-\eta}}{N}+\frac{t^{3/2-\eta/2}}{N^{1/2}}\right).
$$
\end{Lemma}

\begin{proof}
We use the congruences to count the number of $(m,m')$ in \eqref{omega-bound-2}. This comes out to be dominated by 
$$
O((d_2,q_2'd_2'n_1)(d_2',n_2)(1+M_1/d_2)(1+M_1/d_2')).
$$ 
It follows that the contribution of this case to $\Omega_{\neq 0}$ is dominated by
\begin{align*}
\frac{N_0q_1^3r}{n_1^3M_1^{1/2}t}\mathop{\sum\sum}_{d_2,d_2'}\:d_2d_2'\mathop{\sum\sum}_{\substack{q_2\sim C/q_1d_2\\q_2'\sim C/q_1d_2'}}\:\mathop{\sum}_{1\leq n_2\ll  N_2} (d_2,q_2'd_2'n_1)(d_2',n_2) \left(1+\frac{M_1}{d_2}\right)\left(1+\frac{M_1}{d_2'}\right).
\end{align*}
Summing over $n_2$ and $q_2$ we arrive at
\begin{align*}
\frac{N_0q_1^2rCN_2}{n_1^3M_1^{1/2}t}\mathop{\sum\sum}_{d_2,d_2'}\:d_2'\mathop{\sum}_{\substack{q_2'\sim C/q_1d_2'}}\: (d_2,q_2'd_2'n_1)\left(1+\frac{M_1}{d_2}\right)\left(1+\frac{M_1}{d_2'}\right).
\end{align*}
Next summing over $d_2$ we get
\begin{align*}
\frac{N_0q_1^2rCN_2}{n_1^3M_1^{1/2}t}\mathop{\sum}_{d_2'}\:d_2'\mathop{\sum}_{\substack{q_2'\sim C/q_1d_2'}}\: \left(\frac{C}{q_1}+M_1\right)\left(1+\frac{M_1}{d_2'}\right).
\end{align*}
Executing the remaining sums we get
\begin{align}
\label{first-bd}
\frac{N_0q_1rC^2N_2}{n_1^3M_1^{1/2}t}\;\left(\frac{C}{q_1}+M_1\right)^2\ll \frac{q_1r}{n_1^3}\left(\frac{N_0N_2C^4}{M_1^{1/2}tq_1^2}+\frac{N_0N_2C^2M_1^{3/2}}{t}\right).
\end{align}
Suppose $M_1\asymp (tC)^2/N$ or $M_1\gg C/q_1$, then when the above bound is substituted for $\Omega$ in \eqref{cauchy} we get the bound 
$$
r^{1/2}t^{3/2}N^{1/2}\left(\frac{t^{3-\eta}}{N}+\frac{t^{3/2-\eta/2}}{N^{1/2}}\right)
$$
for $C\ll t^{1+\varepsilon}$. In the complementary range when $M_1\ll C/q_1$  and $M_1$ is not of size $(tC)^2/N$, then $N_0\asymp (Ct)^3r/N$. In this case we adopt a different strategy for counting. (Let $d_2\sim D \ll D'\sim d_2'$.) In this case $q_2d_2n_1+m'n_2\ll Cn_1/q_1+M_1N_2\ll Cn_1/q_1+N/n_1q_1^2t^2$. Writing $q_2d_2n_1+m'n_2=-d_2'h$ we see that $h\ll Cn_1/q_1D'+N/n_1q_1^2t^2D':=H$. With this we transform \eqref{omega-bound-2} to
\begin{align}
\label{omega-bound-3}
\frac{N_0q_1^3r}{n_1^3M_1^{1/2}}\mathop{\sum\sum}_{d_2,d_2'}\:d_2d_2'\mathop{\sum\sum}_{\substack{h\ll H\\q_2'\sim C/q_1d_2'}}\:\mathop{\mathop{\sum\sum}_{m,m'\sim M_1}\sum_{n_2\in \mathbb{Z}-\{0\}}}_{\substack{q_2'd_2'n_1+mn_2\equiv 0\bmod{d_2}\\ hd_2'+m'n_2\equiv 0\bmod{d_2}}}\:|\mathfrak{I}|.
\end{align}
Using the second congruence we count the number of $d_2'$ which comes out to be $O((d_2,m'n_2)D'/D)$. The first congruence gives us the number of $m$ which comes out to be $O((n_2,d_2)(1+M_1/D))$. It follows that \eqref{omega-bound-3} is dominated by  
\begin{align*}
\frac{N_0q_1^3r}{n_1^3M_1^{1/2}t}\mathop{\sum}_{d_2\sim D}\:D^{'2}\mathop{\sum\sum}_{\substack{h\ll H\\q_2'\sim C/q_1D'}}\:\mathop{\sum}_{m'\sim M_1}\sum_{0<n_2\ll N_2}\:(m'n_2,d_2)(n_2,d_2)\left(1+\frac{M_1}{D}\right).
\end{align*}
Then summing over $n_2$, $m'$ and $d_2$ we arrive at
\begin{align*}
\frac{N_0q_1^3r}{n_1^3M_1^{1/2}t}\:M_1N_2D\:D^{'2}\mathop{\sum\sum}_{\substack{h\ll H\\q_2'\sim C/q_1D'}}\:\left(1+\frac{M_1}{D}\right),
\end{align*}
which is dominated by
\begin{align}
\label{second-bd}
\frac{N_0q_1r}{n_1^3M_1^{1/2}t}\:M_1N_2C\left(Cn_1+\frac{N}{n_1q_1t^2}\right)\:\left(D+M_1\right).
\end{align}
Now we substitute $D\ll C/q_1$, $M_1\ll C/q_1$ and $C\ll t^{1+\varepsilon}$. When the above bound is substituted in place of $\Omega$ in \eqref{cauchy} we get the bound 
$$
r^{1/2}t^{3/2}N^{1/2}\left(\frac{t^{3/2-\eta}}{N^{1/2}}+t^{-\eta/2}\right).
$$
This is dominated by the previous bound. The lemma follows.
\end{proof}

\subsection{The generic case}
It now remains to tackle the case where $C\gg t^{1+\varepsilon}$ and $n_2\neq 0$. 

\begin{Lemma}
\label{lem:big-c}
The contribution of  $q\sim C\gg t^{1+\varepsilon}$, and $n_2\neq 0$  to \eqref{cauchy} is bounded by
$$
\Sigma_{\neq 0, \:\text{generic}}\ll r^{1/2}t^{3/2}N^{1/2}\;\frac{t^{3\eta/4}N^{1/4}}{t^{11/12}}\ll N^{1/2}t^{3/2 -1/6+3\eta/4+\theta/2}.
$$
\end{Lemma}

\begin{proof}
In this case we need a better bound for $\mathfrak{I}$. To this end we seek to apply stationary phase analysis to the integral $I(\dots)$ in \eqref{integral}, namely
\begin{align*}
\int_0^\infty U(y)e\left(-\frac{t}{2\pi}\log y \pm A\sqrt{y}\pm B (y+u))^{1/3}\right)\mathrm{d}y,
\end{align*}
where $A=2\sqrt{mN}/q$ and $B=3(Nn_1^2n_2)^{1/3}/qr^{1/3}$.
Since $C\gg t^{1+\varepsilon}$, from \eqref{voronoi2} we conclude that we have plus sign with $A$ and that $A\asymp t$. From \eqref{voronoi3} we conclude that $B\ll t^{1-\eta/2}$. (Otherwise the integrals in \eqref{voronoi2} and \eqref{voronoi3} are negligibly small.) As such the stationary point can be written as $y_0+y_1+y_2+\dots$ with $y_i\ll (B/t)^i$. Explicit calculation yields
$$
y_0=\left(\frac{t}{\pi A}\right)^2,\;\;\; y_1=\mp \frac{4\pi B}{3t}\left(\frac{t}{\pi A}\right)^{8/3},
$$
and in general $y_k=f_k(t,A)(B/t)^k$ for some function $f_k$. 
It follows that $I(m,n_1^2n_2,q)$ is essentially given by
\begin{align*}
\frac{1}{t^{1/2}}y_0^{-it}\:e\left(Bg_1(A)+B^2g_2(A)+O\left(\frac{B^3}{t^2}\right)\right)
\end{align*}
where $g_1(A)=\mp t^{2/3}/3(\pi A)^{2/3}\ll 1$ and $g_2(A)\ll 1/t$. Also note that $B\asymp (NN_0)^{1/3}/qr^{1/3}$. It follows that the integral $\mathfrak{I}$ is given by 
\begin{align*}
\frac{1}{t}\int W(y) &e\left(\left(Bg_1(A)-B'g_1(A')\right)+\left(B^2g_2(A)-B^{'2}g_2(A')\right)+O\left(\frac{NN_0}{C^3rt^2}\right)\right)\\
&\times e\left(-\frac{N_0n_1n_2y}{q_2q_2'q_1r}\right)\mathrm{d}y
\end{align*}
where in $B$, $B'$ we replace $n_1^2n_2$ by $N_0y$. Since $n_2\neq 0$ we get 
$$
\frac{N_0n_1n_2y}{q_2q_2'q_1r}\gg \frac{N_0n_1}{C^2r}\gg t^\varepsilon \frac{NN_0}{C^3rt^2}
$$
as $C\gg t^{1+\varepsilon}$ and $N\ll t^{3+\varepsilon}$. Making a change of variable $y=z^3$ and using the third derivative bound for the exponential integral we get
\begin{align*}
\mathfrak{I}\ll \frac{1}{t}\left(\frac{q_2q_2'q_1r}{N_0n_1n_2}\right)^{1/3}\ll \frac{Cr^{1/3}t^{2/3}}{t(NN_0)^{1/3}}.
\end{align*}
In our bounds for $\Omega$ (see \eqref{first-bd} and \eqref{second-bd}), we had the factor $N_0N_2$ which boils down to $C(NN_0)^{1/3}r^{2/3}/n_1q_1$ by substituting the value of $N_2$. Now when we incorporate the new bound for the integral, this factor is replaced by $C^2r/n_1q_1t^{1/3}$. Making this replacement in the proof of Lemma~\ref{lem:small-c}, we get Lemma~\ref{lem:big-c}. 
\end{proof}

We now pull together the bounds from Lemma~\ref{lem:zero}, Lemma~\ref{lem:small-c} and Lemma~\ref{lem:big-c} to get that
$$
\frac{S_r(N)}{N^{1/2}t^{3/2}} \ll r^{1/3}\:(t^{-1/2+\eta/2}+t^{-3\eta/2})+r^{1/2}\left(\frac{t^{3-\eta}}{N}+\frac{t^{3/2-\eta/2}}{N^{1/2}}\right)+r^{1/2}\frac{t^{3\eta/4}N^{1/4}}{t^{11/12}},
$$
where $t^{3-\theta}/r^2<N<t^3/r^2$. It follows that
$$
\frac{S_r(N)}{N^{1/2}t^{3/2}} \ll t^{-1/2+\eta/2+\theta/3}+t^{-3\eta/2+\theta/3}+t^{7\theta/2-\eta}+t^{2\theta-\eta/2}+t^{-1/6+3\eta/4},
$$
for $r\ll t^\theta$.
Hence we need $\eta>7\theta/2$, and consequently the third term dominates the second and the fourth terms. Also we see that the last term dominates the first. Hence the above bound reduces to
$$
\frac{S_r(N)}{N^{1/2}t^{3/2}} \ll t^{7\theta/2-\eta}+t^{-1/6+3\eta/4}.
$$
The optimal choice for $\eta$ is given by $\eta=2\theta+2/21$. Plugging this in \eqref{new-afe} we get that 
$$
L(1/2+it,\pi\times f)\ll t^{3/2+3\theta/2-2/21}+t^{3/2-\theta/2},
$$
and with the optimal choice $\theta=1/21$ we obtain the bound given in Theorem~\ref{mthm}.
 \\

%===================================================================
%===================================================================

%========================================================================================================================

\end{document}